\documentclass[10pt]{article}
\usepackage{mathtools}
\usepackage{graphicx}
\usepackage{amsmath,amsfonts,latexsym,amscd,amssymb,theorem}
\usepackage{graphicx}
\usepackage[ansinew]{inputenc}
\usepackage{epsfig}
\usepackage{amsmath}
\usepackage{amssymb}
\usepackage{afterpage}
\usepackage{selinput}

\newcommand{\ba}{\begin{eqnarray}}
\newcommand{\ea}{\end{eqnarray}}

\newtheorem{thm}{Theorem}[section]

\newtheorem{conjecture}{Conjecture}

\newtheorem{theorem}[thm]{Theorem}

\newtheorem{lemma}[thm]{Lemma}

\newtheorem{corollary}[thm]{Corollary}
\newtheorem{remark}[thm]{Remark}

\makeatletter
\newcommand*{\rom}[1]{\expandafter\@slowromancap\romannumeral #1@}
\makeatother

\begin{document}
\title{\textbf{Tournaments and the Erd\"{o}s-Hajnal Conjecture}}
\maketitle


\begin{center}
\author{Soukaina ZAYAT \footnote{Department of Mathematics Faculty of Sciences I, Lebanese University, Beirut - Lebanon.\vspace{1.5mm} (soukaina.zayat.96@outlook.com)}, Salman GHAZAL \footnote{Department of Mathematics Faculty of Sciences I, Lebanese University, Beirut - Lebanon.\\ salman.ghazal@ul.edu.lb}\footnote{Department of Mathematics and Physics, School of Arts and Sciences, Beirut International University, Beirut - Lebanon. \vspace{2mm} salman.ghazal@liu.edu.lb}}
\end{center}

\begin{abstract}
The celebrated Erd\"{o}s-Hajnal conjecture states that for every undirected graph $H$ there exists $ \epsilon(H) > 0 $ such that every undirected graph on $ n $ vertices that does not contain $H$ as an induced subgraph contains a clique or a stable set of size at least $ n^{\epsilon(H)} $. This conjecture has a directed equivalent version stating that for every tournament $H$ there exists $ \epsilon(H) > 0 $ such that every $H$-free $n$-vertex tournament $T$ contains a transitive subtournament of order at least $ n^{\epsilon(H)} $. This conjecture is proved for few infinite families of tournaments. In this paper we construct a new infinite family of tournaments $-$ the family of so-called flotilla-galaxies and we prove the correctness of the conjecture for every flotilla-galaxy tournament. 
\end{abstract}

\section{Introduction}
Let $ G $ be an undirected graph. We denote by $ V(G) $ the set of its vertices and by $ E(G) $ the set of its edges. We call $ \mid$$G$$\mid$ $=$ $ \mid$$V(G)$$\mid$ the \textit{size} of $G$. Let $X \subseteq V(G)$, the \textit{subgraph of} $G$ \textit{induced by} $X$ is denoted by $G$$\mid$$X$, that is the graph with vertex set $X$, in which $x,y \in X$ are adjacent if and only if they are adjacent in $G$. A \textit{clique} in $G$ is a set of pairwise adjacent vertices and a \textit{stable set} in $G$ is a set of pairwise nonadjacent vertices. For an undirected graph $H$, we say that $G$ is $H$-$free$ if no induced subgraph of $G$ is isomorphic to $H$. A \textit{digraph} is a pair $D=(V,E)$ of sets such that $E\subset V \times V$, and such that for every $(x,y)\in E$ we must have $(y,x)\notin E$, in particular if $(x,y)\in E$ then $x \neq y$. $E$ is the arc set and $V$ is the vertex set and they are denoted by $E(D)$ and $V(D)$ respectively. We say that $D'$ is a \textit{subdigraph} of a digraph $D$ if $V(D') \subseteq V(D)$ and $E(D') \subseteq E(D)$. We say that $D$ contains a copy of $D'$ if $D'$ is isomorphic to a subdigraph of $D$. A \textit{tournament} is a directed graph (digraph) such that for every pair $u$ and $v$ of vertices, exactly one of the arcs $(u,v)$ or $(v,u)$ exists. A tournament is \textit{transitive} if it contains no directed cycle. We denote by $T^{c}$ the tournament obtained from the tournament $T$ by reversing the directions of all the arcs of $T$. Let $T$ be a tournament. We denote its vertex set by $V(T)$ and its arc set by $E(T)$, and we write $\mid$$T$$\mid$ for $\mid$$V(T)$$\mid$. We say that $\mid$$T$$\mid$ is the size of $T$. Let $X \subseteq V(T)$, the \textit{subtournament of} $T$ \textit{induced by} $X$ is denoted by $T$$\mid$$X$, that is the tournament with vertex set $X$, such that for $x,y \in X$, $(x,y) \in E(T$$\mid$$ X)$ if and only if $(x,y) \in E(T)$. If $(u,v)\in E(T)$ then we say that $u$ is \textit{adjacent to} $v$ (alternatively: $v$ is an \textit{outneighbor} of $u$) and we write $u\rightarrow v$, also we say that $v$ is \textit{adjacent from} $u$ (alternatively: $u$ is an \textit{inneighbor} of $v$) and we write $v\leftarrow u$. For two sets of vertices $V_{1},V_{2}$ of $T$ we say that $V_{1}$ is \textit{complete to} (resp. \textit{from}) $V_{2}$ if every vertex of $V_{1}$ is adjacent to (resp. from) every vertex of $V_{2}$, and we write $V_{1} \rightarrow V_{2}$ (resp. $V_{1} \leftarrow V_{2}$). We say that a vertex $v$ is complete to (resp. from) a set $V$ if $\lbrace v \rbrace$ is complete to (resp. from) $V$ and we write $v \rightarrow V$ (resp. $v \leftarrow V$).  Given a tournament $H$, we say that $T$ \textit{contains} $H$ if $H$ is isomorphic to $T$$\mid$$X$ for some $X \subseteq V(T)$. If $T$ does not contain $H$, we say that $T$ is $H$-$free$. 
\vspace{1.5mm}\\
Erd\"{o}s and Hajnal proposed the following conjecture \cite{jhp} (EHC):
\begin{conjecture} For any undirected graph $H$ there exists $ \epsilon(H) > 0 $ such that any $ H$-free undirected graph with $n$ vertices contains a clique or a stable set of size at least $ n^{\epsilon(H)}. $
\end{conjecture}
In 2001 Alon et al. proved \cite{fdo} that Conjecture $1$ has an equivalent directed version, as follows:
\begin{conjecture} \label{a} For any tournament $H$ there exists $ \epsilon(H) > 0 $ such that every $ H $-free tournament with $n$ vertices contains a transitive subtournament of size at least $ n^{\epsilon(H)}. $
\end{conjecture}

A tournament $H$ \textit{satisfies the Erd\"{o}s-Hajnal Conjecture (EHC)} (equivalently: $H$ has the \textit{Erd\"{o}s-Hajnal property}) if there exists $ \epsilon(H) > 0 $ such that every $ H $-free tournament $T$ with $n$ vertices contains a transitive subtournament of size at least $ n^{\epsilon(H)}. $\vspace{2.5mm}\\
The Erd\"{o}s-Hajnal property is a \textit{hereditary property}, i.e if a tournament $H$ has the Erd\"{o}s-Hajnal property then all its subtournaments also have the Erd\"{o}s-Hajnal property.\vspace{3mm}

 The Erd\"{o}s-Hajnal conjecture is known for all tournaments on at most six vertices except one \cite{polll,bnmm}, and for few infinite classes of tournaments \cite{fdo,polll,kg}.\\ A set of vertices $S \subseteq V(H)$ of a tournament $H$ is called \textit{homogeneous} if for every $v \in V(H)$$ \backslash S$ the following holds: either for all $w \in S$ we have: $(w,v)$ is an arc or for all $w \in S$ we have: $(v,w)$ is an arc. A homogeneous set $S$ is called \textit{nontrivial} if $ \mid $$S$$ \mid > 1 $ and $S \neq V(H)$. A tournament is called \textit{prime} if  it does not have nontrivial homogeneous sets, else we call it \textit{not prime}. An outstanding result of Alon et al. \cite{fdo} that is applied to tournaments shows that if the conjecture holds for the tournaments $H$ and $H'$ then it also holds for the tournament obtained by substituting $H'$ into a vertex of $H$.  All tournaments that can be constructed by the substitution procedure have nontrivial homogeneous sets. Then a tournament is prime if and only if it is not obtained from smaller tournaments by substitution. And so the result of \cite{fdo} implies the following theorem that explains why prime tournaments plays a central role:
\begin{theorem}
 If the Erd\"{o}s-Hajnal conjecture is not true then the smallest counterexample is prime.
\end{theorem} 
Let $ \theta = (v_{1},...,v_{n}) $ be an ordering of the vertex set $V(D)$ of an $ n- $vertex digraph $D$.
An arc $ (v_{i},v_{j})\in E(D) $ is a \textit{backward arc of $D$ under} $ \theta $ if $ i > j $. We say that a vertex $ v_{j} $ is \textit{between} two vertices $ v_{i},v_{k} $ under $ \theta = (v_{1},...,v_{n}) $ if $ i < j < k $ or $ k < j < i $. 
 The graph of backward arcs under $ \theta $, denoted by $ B(D,\theta) $, is the undirected graph that has vertex set $V(D)$, and $ v_{i}v_{j} \in E(B(D,\theta)) $ if and only if $ (v_{i},v_{j}) $ or $ (v_{j},v_{i}) $ is a backward arc of $D$ under $ \theta $. A tournament $S$ on $p$ vertices with $V(S)= \lbrace u_{1},u_{2},...,u_{p}\rbrace$ is a \textit{right star} (resp. \textit{left star}) (resp. \textit{middle star}) if there exist an ordering $\theta^{*} = (u_{1},u_{2},...,u_{p})$ of its vertices such that the backward arcs of $S$ under $\theta^{*}$ are $(u_{p},u_{i})$ for $i=1,...,p-1$ (resp. $(u_{i},u_{1})$ for  $i=2,...,p$) (resp. $(u_{i},u_{r})$ for $i= r+1,...,p$ and $(u_{r},u_{i})$ for $i=1,...,r-1$, where $2\leq r\leq p-1$). In this case we write $S = \lbrace u_{1},u_{2},...,u_{p}\rbrace$ and we call $\theta^{*} = (u_{1},u_{2},...,u_{p})$ a \textit{right star ordering} (resp. \textit{left star ordering}) (resp. \textit{middle star ordering}) of $S$, $u_{p}$ (resp. $u_{1}$) (resp. $u_{r}$) the \textit{center of} $S$, and $u_{1},...,u_{p-1}$ (resp. $u_{2},...,u_{p}$) (resp. $u_{1},...,u_{r-1},u_{r+1},...,u_{p}$) the \textit{leaves of} $S$. A \textit{star} is a left star or a right star or a middle star. A \textit{star ordering} is a left star ordering or a right star ordering or a middle star ordering. Note that in the case $p=2$ we may choose arbitrarily any one of the two vertices to be the center of the star, and the other vertex is then considered to be the leaf. A \textit{frontier star} is a left star or a right star (note that a frontier star is not a middle star, a frontier star is either left or right).  
   A \textit{star} $S=\lbrace v_{i_{1}},...,v_{i_{t}}\rbrace$ \textit{of $D$ under $\theta$} (where $i_{1}<...<i_{t}$) is the subdigraph of $D$ induced by $\lbrace v_{i_{1}},...,v_{i_{t}}\rbrace$ such that $S$ is a star and $S$ has the star ordering $ (v_{i_{1}},...,v_{i_{t}})$ under $\theta$ (i.e $(v_{i_{1}},...,v_{i_{t}})$ is the restriction of $\theta$ to $V(S)$ and $ (v_{i_{1}},...,v_{i_{t}})$ is a star ordering of $S$).\vspace{3.5mm}

In \cite{polll} Berger et al. constructed a new infinite family of tournaments (so-called \textit{galaxies}) containing infinitely many prime tournaments that satisfy Conjecture \ref{a}.\\
A tournament $T$ is a \textit{galaxy} if there exists an ordering $\theta$ of its vertices such that $V(T)$ is the disjoint union of $V(Q_{1}),...,V(Q_{l}),X$ where $Q_{1},...,Q_{l}$ are the frontier stars of $T$ under $\theta$, and for every $x\in X$, $\lbrace x \rbrace$ is a singleton component of $B(T,\theta)$, and no center of a star is between leaves of another star under $\theta$. In this case we also say that $T$ \textit{is a \textit{galaxy} under $\theta$}. If $X=\phi$, we say that $T$ is a \textit{regular galaxy under $\theta$} (see Figure \ref{fig:galaxy}). 
\begin{theorem} \cite{polll}
Every galaxy satisfies the Erd\"{o}s-Hajnal conjecture.
\end{theorem}
\begin{figure}[h]
	\centering
	\includegraphics[width=0.45\linewidth]{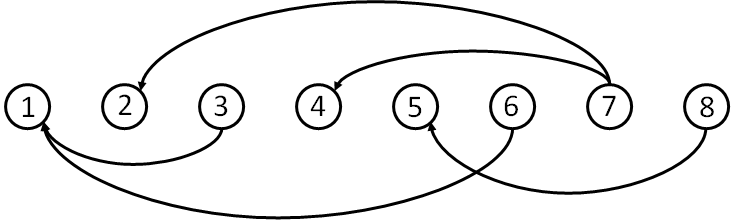}
	\caption{Galaxy under $(1,2,...,8)$ consisting of one left star and two right stars. All the non drawn arcs are forward.}
	\label{fig:galaxy}
\end{figure}
And in $2015$ Choromanski extends the family of galaxies to \textit{constellations} by making the condition concerning the centers of the stars much weaker than that in galaxies, but still middle stars not allowed (constellations are fully characterized in \cite{kg}) and proved the following theorem:
\begin{theorem}\label{constellation}\cite{kg}
Every constellation satisfies the Erd\"{o}s-Hajnal conjecture.
\end{theorem}
More recently we extend the family of galaxies to \textit{galaxies with spiders} \cite{sg} in which we allow middle stars to exist under some conditions and we replace the condition concerning centers of stars by a weaker one:
\begin{theorem}\cite{sg}
Every galaxy with spiders satisfy the Erd\"{o}s-Hajnal conjecture.
\end{theorem}
A tournament $T$ is a \textit{nebula} if there exists an ordering $\theta$ of its vertices such that $V(T)$ is the disjoint union of $V(Q_{1}),...,V(Q_{l}),X$ where $Q_{i}$ is a star of $T$ under $\theta$ ($Q_i$ may be middle star)  for $i=1,...,l$, and for every $x\in X$, $\lbrace x \rbrace$ is a singleton component of $B(T,\theta)$  (note that there is no condition concerning the location of the centers of the stars and middle stars). In this case say that $\theta$ is a \textit{nebula ordering of $T$}.\vspace{1mm}\\

Unfortunately showing that every nebula satisfies the Erd\"{o}s-Hajnal conjecture is still wide open problem and considered very hard. The only known results concerning nebulas are for galaxies, constellations and galaxies with spiders. \\
 On the other hand there exist infinitely many tournaments with no nebula ordering and not known to satisfy $EHC$. That motivates us to work on new configuration of backward arcs. Our first result concerning tournaments with no nebula ordering is for an infinite class of tournaments $-$ the so-called \textit{asterisms} \cite{sg}. To prove $EHC$ for asterisms we introduce a very powerful tool $-$ the so-called \textit{"Corresponding Digraph"} that turns to be very useful in \textit{flotilla-galaxies}, the infinite class treated in this paper. Flotilla-galaxy tournaments have no nebula ordering, instead a flotilla-galaxy has a special backward arc configuration consisting of disjoint group of $4$-vertex paths and stars (note that middle stars on three vertices are allowed). Middle stars and $4$-vertex paths are considered of special interest and very hard to treat. That motivates us to work on backward arc configuration consisting of such structures. \vspace{3mm}\\
 The main result of this paper is the following:
\begin{theorem}\label{generalflotilla-galaxy}
Every flotilla-galaxy satisfies the Erd\"{o}s-Hajnal conjecture.
\end{theorem}
This paper is organized as follows:\vspace{1mm}\\
$\bullet$ In section $2$ we give some properties of $\epsilon$-critical tournaments needed in the proof of the main results in this paper.\\
$\bullet$ In section $3$ we define formally what are flotilla-galaxies, we introduce some definitions and we prove Theorem \ref{generalflotilla-galaxy}. 
   
\section{$\epsilon$-critical tournaments}
Denote by $tr(T)$ the largest size of a transitive subtournament of a tournament $T$. For $X \subseteq V(T)$, write $tr(X)$ for $tr(T$$\mid$$X)$. Let $X, Y \subseteq V(T)$ be disjoint. Denote by $e_{X,Y}$ the number of directed arcs $(x,y)$, where $x \in X$ and $y \in Y$. The directed density from $X$ to $Y$ is defined as $d(X,Y) = \frac{e_{X,Y}}{\mid X \mid.\mid Y \mid} $.\\ We call $T$ $ \epsilon  $-\textit{critical} for $ \epsilon > 0 $ if $tr(T) < $ $ \mid $$T$$ \mid^{\epsilon} $ but for every proper subtournament $S$ of $T$ we have: $tr(S) \geq $ $ \mid $$S$$ \mid^{\epsilon}. $ The following are some properties of $\epsilon$-critical tournaments that we borrow from \cite{polll,bnmm,sss,sg}:

\begin{lemma} \cite{polll} \label{e} For every $N$ $ > 0 $, there exists $ \epsilon(N) > 0 $ such that for every $ 0 < \epsilon < \epsilon(N)$ every $ \epsilon $-critical tournament $T$ satisfies $ \mid $$T$$\mid$ $ \geq N$.
\end{lemma}
\begin{proof}
Since every tournament contains a transitive subtournament of size $2$ so it suffices to take $\epsilon(N) = log_{N}(2)$. $\blacksquare$
\end{proof}
\begin{lemma} \cite{polll} \label{f} Let $T$ be an $ \epsilon $-critical tournament with $\mid$$T$$\mid$ $ =n$ and $\epsilon$,$c,f > 0 $ be constants such that $ \epsilon < log_{c}(1 - f)$. Then for every $A \subseteq V(T)$ with $ \mid $$A$$ \mid$ $ \geq cn$ and every transitive subtournament $G$ of $T$ with $ \mid $$G$$ \mid$ $\geq f.tr(T)$ and $V(G) \cap A = \phi$, we have: $A$ is not complete from $V(G)$ and $A$ is not complete to $V(G)$.
\end{lemma}
\begin{lemma} \cite{polll} \label{v} Let $T$ be an $ \epsilon $-critical tournament with $\mid$$T$$\mid$ $= n$ and $\epsilon$,$c > 0 $ be constants such that $ \epsilon < log_{\frac{c}{2}}(\frac{1}{2}). $ Then for every two disjoint subsets $X, Y \subseteq V(T)$ with $ \mid $$X$$ \mid$ $ \geq cn, \mid $$Y$$ \mid$ $ \geq cn $ there exist an integer $k \geq \frac{cn}{2} $ and vertices $ x_{1},...,x_{k} \in X $ and $ y_{1},...,y_{k} \in Y $ such that $ y_{i} $ is adjacent to $ x_{i} $ for $i = 1,...,k. $
\end{lemma}
\begin{lemma}\cite{sg}\label{r}
Let $f_{1},...,f_{m},c,\epsilon > 0$ be constants, where $0 <  f_{1},...,f_{m},c < 1$ and $0 < \epsilon < log_{\frac{c}{2m}}(1-f_{i})$ for $i=1,...,m$. Let $T$ be an $ \epsilon $-critical tournament with $\mid$$T$$\mid$ $ =n$, and   let $S_{1},...,S_{m}$ be m disjoint transitive subtournaments of $T$ with $ \mid $$S_{i}$$ \mid$ $\geq f_{i}.tr(T)$ for $i=1,...,m$. Let $A \subseteq V(T) \backslash (\bigcup_{i=1}^{m} V(S_{i}))$ with $ \mid $$A$$ \mid$ $ \geq cn$.
Then there exist vertices $s_{1},...,s_{m},a$ such that $a\in A$, $s_{i}\in S_{i}$ for $i=1,...,m$, and $\lbrace a \rbrace$ is complete to $\lbrace s_{1},...,s_{m} \rbrace$. Similarly there exist vertices $u_{1},...,u_{m},b$ such that $b\in A$, $u_{i}\in S_{i}$ for $i=1,...,m$, and $\lbrace b \rbrace$ is complete from $\lbrace u_{1},...,u_{m} \rbrace$.
\end{lemma}
\begin{proof}
We will prove only the first statement  because the latter can be proved analogously.
Let $A_{i} \subseteq A$ such that $A_{i}$ is complete from $S_{i}$ for $i = 1,...,m$. Let $1\leq j \leq m$. If $\mid$$A_{j}$$\mid \geq \frac{\mid A \mid}{2m}\geq \frac{c}{2m}n$, then this will contradicts Lemma \ref{f} since $\mid$$S_{j}$$\mid \geq f_{j}tr(T)$ and $\epsilon < log_{\frac{c}{2m}}(1-f_{j})$. Then $\forall i \in \lbrace 1,...,m \rbrace$, $\mid$$A_{i}$$\mid < \frac{\mid A \mid}{2m}$. Let $A^{*} = A\backslash (\bigcup_{i=1}^{m}A_{i})$, then $\mid$$A^{*}$$\mid > \mid $$A$$ \mid-m.\frac{\mid A \mid}{2m} \geq \frac{\mid A \mid}{2}$. Then $A^{*} \neq \phi$. Fix $a\in A^{*}$. So there exist vertices $s_{1},...,s_{m}$ such that $s_{i}\in S_{i}$ for $i=1,...,m$, and $\lbrace a \rbrace$ is complete to $\lbrace s_{1},...,s_{m} \rbrace$.  $\blacksquare$\vspace{3mm}\\ 
\end{proof}
The proof of the following lemma is completely analogous to the proof of Lemma \ref{r}.
\begin{lemma} \label{middle}
Let $f_{1},f_{2},c,\epsilon > 0$ be constants, where $0 <  f_{1},f_{2},c < 1$ and $0 < \epsilon < min\lbrace log_{\frac{c}{4}}(1-f_{1}), log_{\frac{c}{4}}(1-f_{2})\rbrace$. Let $T$ be an $ \epsilon $$-$critical tournament with $\mid$$T$$\mid$ $ =n$, and   let $S_{1},S_{2}$ be two disjoint transitive subtournaments of $T$ with $ \mid $$S_{1}$$ \mid$ $\geq f_{1}.tr(T)$ and $ \mid $$S_{2}$$ \mid$ $\geq f_{2}.tr(T)$. Let $A \subseteq V(T) \backslash (V(S_{1})\cup V(S_{2}))$ with $ \mid $$A$$ \mid$ $ \geq cn$. Then there exist vertices $a,s_{1},s_{2}$ such that $a\in A, s_{1}\in S_{1}, s_{2}\in S_{2}$ and $s_{1} \leftarrow a \leftarrow s_{2}$. 
\end{lemma}
\begin{proof}
Let $A_{1}$ be the set of vertices of $A$ that are complete from $S_{1}$, and let $A_{2}$ be the set of vertices of $A$ that are complete to $S_{2}$. Assume that $\mid$$A_{1}$$\mid$ $\geq \frac{\mid A \mid}{4} \geq \frac{c}{4}n$. Since $\epsilon < log_{\frac{c}{4}}(1-f_{1})$, then Lemma \ref{f} implies that $S_{1}$ is not complete to $A_{1}$, a contradiction. Then $\mid$$A_{1}$$\mid$ $ < \frac{\mid A \mid}{4}$. Similarly we prove that $\mid$$A_{2}$$\mid$ $ < \frac{\mid A \mid}{4}$. Now let $A^{*} = A\backslash (A_{1} \cup A_{2})$, then $\mid$$A^{*}$$\mid$ $> \frac{\mid A \mid}{2}$. Then $A^{*}\neq \phi$. Fix $a\in A^{*}$. So $\exists s_{1}\in S_{1}$ and $\exists s_{2}\in S_{2}$ such that $s_{1} \leftarrow a \leftarrow s_{2}$. $\blacksquare$ 
\end{proof}
\begin{lemma}\cite{sss} \label{s}
Let $f,c,\epsilon > 0$ be constants, where $0 <  f,c < 1$ and $0 < \epsilon < min\lbrace log_{\frac{c}{2}}(1-f), log_{\frac{c}{4}}(\frac{1}{2})\rbrace$. Let $T$ be an $ \epsilon $$-$critical tournament with $\mid$$T$$\mid$ $ =n$, and   let $S_{1},S_{2}$ be two disjoint transitive subtournaments of $T$ with $ \mid $$S_{1}$$ \mid$ $\geq f.tr(T)$ and $ \mid $$S_{2}$$ \mid$ $\geq f.tr(T)$. Let $A_{1},A_{2}$ be two disjoint subsets of $V(T)$ with $ \mid $$A_{1}$$ \mid$ $\geq cn$, $ \mid $$A_{2}$$ \mid$ $\geq cn$, and $A_{1},A_{2} \subseteq V(T) \backslash (V(S_{1})\cup V(S_{2}))$. Then there exist vertices $a,x,s_{1},s_{2}$ such that $a\in A_{1}, x\in A_{2}, s_{1}\in S_{1}, s_{2}\in S_{2}$, $\lbrace a,s_{1}\rbrace \leftarrow x$, and $a\leftarrow s_{2}$. 
\end{lemma}
\begin{proof}
 Let $A_{1}^{*} = \lbrace a \in A_{1}; \exists s \in S_{2}$ and $a \leftarrow s \rbrace$ and let $A_{2}^{*} = \lbrace x \in A_{2}; \exists v \in S_{1}$ and $v \leftarrow x \rbrace$. Then $A_{1}\backslash A_{1}^{*}$ is complete to $S_{2}$ and $A_{2}\backslash A_{2}^{*}$ is complete from $S_{1}$. Now assume that $\mid$$A_{1}^{*}$$\mid$ $< \frac{\mid A_{1} \mid}{2}$, then $\mid$$A_{1}\backslash A_{1}^{*}$$\mid$ $\geq \frac{\mid A_{1} \mid}{2} \geq \frac{c}{2}n$. Since $\mid$$ S_{2}$$\mid$ $\geq f.tr(T)$ and since $\epsilon < log_{\frac{c}{2}}(1-f)$, then Lemma \ref{f} implies that $A_{1}\backslash A_{1}^{*}$ is not complete to $S_{2}$, a contradiction. Then $\mid$$A_{1}^{*}$$\mid$ $\geq \frac{\mid A_{1} \mid}{2} \geq \frac{c}{2}n$. Similarly we prove that $\mid$$A_{2}^{*}$$\mid$ $\geq  \frac{c}{2}n$. Now since $\epsilon < log_{\frac{c}{4}}(\frac{1}{2})$, then Lemma \ref{v} implies that $\exists k \geq \frac{c}{4}n$, $\exists a_{1},...,a_{k} \in A_{1}^{*}$, $\exists x_{1},...,x_{k} \in A_{2}^{*}$, such that $a_{i} \leftarrow x_{i}$ for $i = 1,...,k$.  So $\exists s_{1} \in S_{1}$, $\exists s_{2} \in S_{2}$ such that  $\lbrace a_{1},s_{1}\rbrace \leftarrow x_{1}$, and $a_{1}\leftarrow s_{2}$.  $\blacksquare$ 
\end{proof}
\begin{lemma} \cite{polll} \label{b} Let $A_{1},A_{2}$ be two disjoint sets such that $d(A_{1},A_{2}) \geq 1-\lambda$ and let $0 < \eta_{1},\eta_{2} \leq 1$. Let $\widehat{\lambda} = \frac{\lambda}{\eta_{1}\eta_{2}}$. Let $X \subseteq A_{1}, Y \subseteq A_{2}$ be such that $\mid$$X$$\mid$ $\geq \eta_{1} \mid$$A_{1}$$\mid$ and $\mid$$Y$$\mid$ $\geq \eta_{2} \mid$$A_{2}$$\mid$. Then $d(X,Y) \geq 1-\widehat{\lambda}$. 
\end{lemma}
The following is introduced in \cite{bnmm}.\\
Let $ c > 0, 0 < \lambda < 1 $ be constants, and let $w$ be a $ \lbrace 0,1 \rbrace - $ vector of length $ \mid $$w$$ \mid $. Let $T$ be a tournament with $ \mid $$T$$ \mid$ $ = n. $ A sequence of disjoint subsets $ \chi = (S_{1}, S_{2},..., S_{\mid w \mid}) $ of $V(T)$ is a smooth $ (c,\lambda, w)- $structure if:\\
$\bullet$ whenever $ w_{i} = 0 $ we have $ \mid $$S_{i}$$ \mid$ $ \geq cn $ (we say that $ S_{i} $ is a \textit{linear set}).\\
$\bullet$ whenever $ w_{i} = 1 $ the tournament $T$$\mid$$ S_{i} $ is transitive and $ \mid $$S_{i}$$ \mid$ $ \geq c.tr(T) $ (we say that $ S_{i} $ is a \textit{transitive set}).\\
$\bullet$ $ d(\lbrace v \rbrace, S_{j}) \geq 1 - \lambda $ for $v \in S_{i} $ and $ d(S_{i}, \lbrace v \rbrace) \geq 1 - \lambda $ for $v \in S_{j}, i < j $ (we say that $\chi$ is \textit{smooth}).
\begin{theorem} \cite{bnmm} \label{i}
Let $S$ be a tournament, let $w$ be a $ \lbrace 0,1 \rbrace - $vector, and let $ 0 < \lambda_{0} < \frac{1}{2} $ be a constant. Then there exist $ \epsilon_{0}, c_{0} > 0 $ such that for every $ 0 < \epsilon < \epsilon_{0} $, every $ S- $free $ \epsilon $$- $critical tournament contains a smooth $ (c_{0}, \lambda_{0},w)- $structure.
\end{theorem}

Let $(S_{1},...,S_{\mid w \mid})$ be a smooth $(c,\lambda ,w)$$-$structure of a tournament $T$, let $i \in \lbrace 1,...,\mid$$w$$\mid \rbrace$, and let $v \in S_{i}$. For $j\in \lbrace 1,2,...,\mid$$w$$\mid \rbrace \backslash \lbrace i \rbrace$, denote by $S_{j,v}$ the set of the vertices of $S_{j}$ adjacent from $v$ for $j > i$ and adjacent to $v$ for $j<i$.
\begin{lemma}\cite{sss} \label{g} Let $0<\lambda<1$, $0<\gamma \leq 1$ be constants and let $w$ be a $\lbrace 0,1 \rbrace$$-$vector. Let $(S_{1},...,S_{\mid w \mid})$ be a smooth $(c,\lambda ,w)$$-$structure of a tournament $T$ for some $c>0$. Let $j\in \lbrace 1,...,\mid$$w$$\mid \rbrace$. Let $S_{j}^{*}\subseteq S_{j}$ such that $\mid$$S_{j}^{*}$$\mid$ $\geq \gamma \mid$$S_{j}$$\mid$ and let $A= \lbrace x_{1},...,x_{k} \rbrace \subseteq \displaystyle{\bigcup_{i\neq j}S_{i}}$ for some positive integer $k$. Then $\mid$$\displaystyle{\bigcap_{x\in A}S^{*}_{j,x}}$$\mid$ $\geq (1-k\frac{\lambda}{\gamma})\mid$$S_{j}^{*}$$\mid$. In particular $\mid$$\displaystyle{\bigcap_{x\in A}S_{j,x}}$$\mid$ $\geq (1-k\lambda)\mid$$S_{j}$$\mid$.
\end{lemma}
\begin{proof}
The proof is by induction on $k$. without loss of generality assume that $x_{1} \in S_{i}$ and $j<i$. Since $\mid$$S_{j}^{*}$$\mid$ $\geq \gamma \mid$$S_{j}$$\mid$ then by Lemma \ref{b}, $d(S^{*}_{j},\lbrace x_{1}\rbrace) \geq 1-\frac{\lambda}{\gamma}$. So $1-\frac{\lambda}{\gamma} \leq d(S^{*}_{j},\lbrace x_{1}\rbrace) = \frac{\mid S^{*}_{j,x_{1}}\mid}{\mid S_{j}^{*}\mid}$. Then $\mid$$S^{*}_{j,x_{1}}$$\mid$ $\geq (1-\frac{\lambda}{\gamma})$$\mid$$S_{j}^{*}$$\mid$ and so true for $k=1$.\\
Suppose the statement is true for $k-1$. $\mid$$\displaystyle{\bigcap_{x\in A}S^{*}_{j,x}}$$\mid$ $=\mid$$(\displaystyle{\bigcap_{x\in A\backslash \lbrace x_{1}\rbrace}S^{*}_{j,x}})\cap S^{*}_{j,x_{1}}$$\mid$ $= \mid$$\displaystyle{\bigcap_{x\in A\backslash \lbrace x_{1}\rbrace}S^{*}_{j,x}}$$\mid$ $+$ $\mid$$S^{*}_{j,x_{1}}$$\mid$ $- \mid$$(\displaystyle{\bigcap_{x\in A\backslash \lbrace x_{1}\rbrace}S^{*}_{j,x}})\cup S^{*}_{j,x_{1}}$$\mid$ $\geq (1-(k-1)\frac{\lambda}{\gamma})\mid$$S_{j}^{*}$$\mid$ $+$ $(1-\frac{\lambda}{\gamma})\mid$$S_{j}^{*}$$\mid$ $-$ $\mid$$S_{j}^{*}$$\mid$ $= (1-k\frac{\lambda}{\gamma})\mid$$S_{j}^{*}$$\mid$. $\blacksquare$       
\end{proof} 
  
\section{EHC for flotilla-galaxies}

Our paper addresses the problem of middle stars and substructures called \textit{boats}, and prove the conjecture for infinitely many tournaments having boats and middle stars on three vertices under some conditions that we will explain in this section. A \textit{boat} $B = \lbrace x,u,v,y \rbrace$ is a $4$-vertex tournament with $V(B) = \lbrace x,u,v,y \rbrace$ and $E(B) = \lbrace (y,x),(y,u),(v,x),(x$ $,u),(u,v),(v,y) \rbrace$.     \\
In order to define formally the infinite family of flotilla-galaxies we need to define four special tournaments on $7$ vertices obtained from a boat $B =\{1,2,3,4\}$. These tournaments are called \textit{generalized boats}. \\
The \textit{left $\gamma_{1}$-boat} is the tournament obtained from $B$ by adding three extra vertices $5$,$6$ and $7$ and making $5$ adjacent to $\lbrace 3,4,7\rbrace$, $6$ adjacent to $\lbrace 4,5\rbrace$, and $7$ adjacent to $6$. The \textit{right $\gamma_{1}$-boat} is the tournament obtained from $B$ by adding three extra vertices $5$,$6$ and $7$ and making $5$ adjacent from $\lbrace 1,2,7\rbrace$, $6$ adjacent from $\lbrace 1,5\rbrace$, and $7$ adjacent from $6$. The \textit{left $\gamma_{2}$-boat} is the tournament obtained from the left $\gamma_{1}$-boat by reversing the direction of the arc $(4,7)$. The \textit{right $\gamma_{2}$-boat} is the tournament obtained from the right $\gamma_{1}$-boat by reversing the direction of the arc $(7,1)$  (see Figure \ref{fig:gammaboats}).\\
In Figure \ref{fig:gammaboats} we define some crucial orderings of the generalized boats.
 \begin{figure}[h]
	\centering
	\includegraphics[width=1\linewidth]{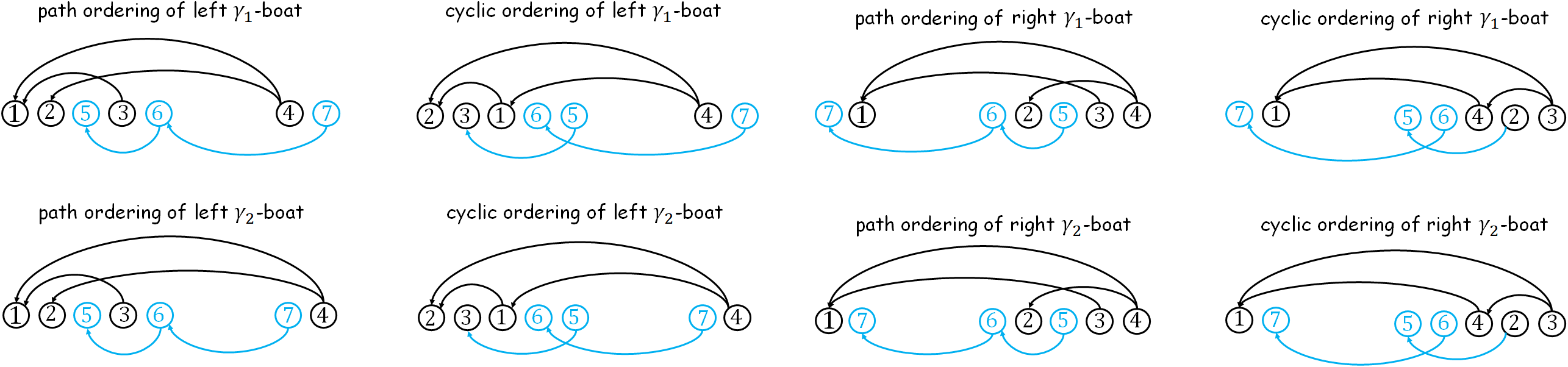}
	\caption{Crucial orderings of the vertices of left and right $\gamma_{1}$-boat, left and right $\gamma_{2}$-boat. All the nondrawn arcs are forward.}
	\label{fig:gammaboats}
\end{figure} 
\\Let $ \theta = (v_{1},...,v_{n}) $ be an ordering of the vertex set $V(T)$ of an $ n$-vertex tournament $T$. A \textit{left $\gamma_{1}$-boat} (resp. \textit{left $\gamma_{2}$-boat}) (resp. \textit{right $\gamma_{1}$-boat}) (resp. \textit{right $\gamma_{2}$-boat})  $B^{\gamma} = \lbrace v_{i_{1}},v_{i_{2}},v_{i_{3}},v_{i_{4}},v_{i_{5}},v_{i_{6}},v_{i_{7}} \rbrace$ \textit{of $T$ under} $\theta$ is an induced subtournament of $T$ with vertex set $\lbrace v_{i_{1}},v_{i_{2}},v_{i_{3}},v_{i_{4}},v_{i_{5}},v_{i_{6}},v_{i_{7}} \rbrace$, such that $B^{\gamma}$ is a left $\gamma_{1}$-boat (resp. left $\gamma_{2}$-boat) (resp. right $\gamma_{1}$-boat) (resp. right $\gamma_{2}$-boat)  and has its path ordering $(v_{i_{1}},v_{i_{2}},v_{i_{3}},v_{i_{4}},v_{i_{5}},v_{i_{6}},v_{i_{7}})$ under $\theta$, and $v_{i_{1}},...,v_{i_{5}}$ are consecutive under $\theta$ and $v_{i_{6}},v_{i_{7}}$ are consecutive under $\theta$ (resp. $v_{i_{1}},...,v_{i_{5}}$ are consecutive under $\theta$ and $v_{i_{6}},v_{i_{7}}$ are consecutive under $\theta$) (rep. $v_{i_{1}},v_{i_{2}}$ are consecutive under $\theta$ and $v_{i_{3}},...,v_{i_{7}}$ are consecutive under $\theta$) (rep. $v_{i_{1}},v_{i_{2}}$ are consecutive under $\theta$ and $v_{i_{3}},...,v_{i_{7}}$ are consecutive under $\theta$).  A \textit{$\gamma$-boat of $T$ under} $\theta$ is a left $\gamma_{1}$-boat or right $\gamma_{1}$$-$boat or left $\gamma_{2}$-boat or right $\gamma_{2}$-boat of $T$ under $\theta$.\vspace{3mm}\\ 
We are ready to define formally the infinite family of flotilla-galaxies.\vspace{2mm}\\
A tournament $T$ is a \textit{flotilla-galaxy} if there exists an ordering $ \theta $ of its vertices such that $V(T)$ is the disjoint union of $V(B^{\gamma}_{1}),...,V(B^{\gamma}_{l}),X$ where $B^{\gamma}_{1},...,B^{\gamma}_{l}$ are the $\gamma$-boats of $T$ under $\theta$, $T$$\mid$$X$ is a galaxy under $\overline{\theta}$ ($\overline{\theta}$ is the restriction of $\theta$ to $X$), and no vertex of a $\gamma$-boat appears in the ordering $\theta$ between leaves of a star of $T$$\mid$$X$ under $\overline{\theta}$. We also say that $T$ \textit{is a flotilla-galaxy under $\theta$}  (see Figure \ref{fig:generalflotillagalaxy}). If for every $x\in X$, $\lbrace x \rbrace$ is a singleton component of $B(T,\theta)$ and the number of the frontier stars of $T$$\mid$$X$ under $\overline{\theta}$ is $l$, then $T$ is called a \textit{regular flotilla-galaxy under $\theta$}.  
\begin{figure}[h]
	\centering
	\includegraphics[width=0.4\linewidth]{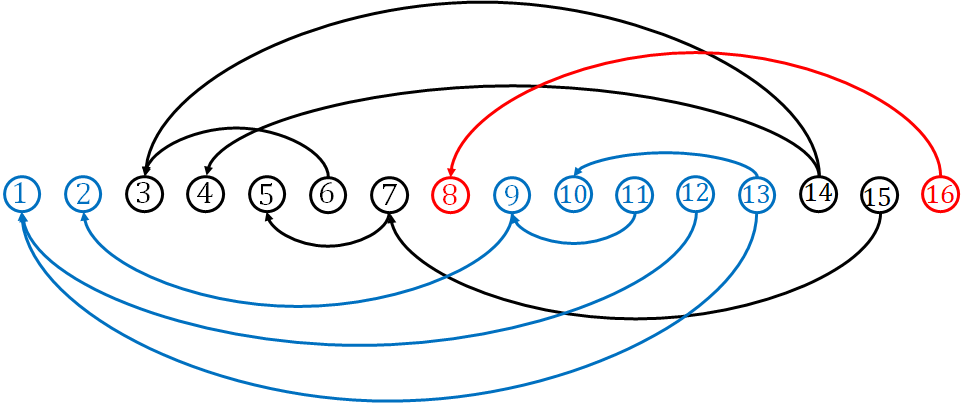}
	\caption{Flotilla-galaxy under $\theta = (1,...,16)$ consisting of one left $\gamma_{1}$-boat, one right $\gamma_{2}$-boat, and one right star. All the nondrawn arcs are forward.}
	\label{fig:generalflotillagalaxy}
\end{figure}
\begin{remark} The main difficulty that is faced in middle stars is looking for its leaves in different transitive sets in a smooth $(c,\lambda ,w)$-structure, which generates doubts about finding middle stars due to having many arcs with unknown orientations.  So we conclude that middle stars are considered very hard.
\end{remark} 
In the following subsection we introduce some tools and technical definitions used to prove EHC for flotilla-galaxies.
\subsection{Definitions and tools}
\subsubsection{Crucial orderings of a generalized flotilla-galaxy}\label{cog}
Let $D$ be a tournament with $7$-vertices $v_{1},...,v_{7}$ and let $\theta_{1}=(v_{1},v_{2},v_{3},v_{4},v_{5},v_{6},v_{7})$ be an ordering of $V(D)$. Let \textit{operation $1$} be the permutation of the vertices $v_{1},...,v_{7}$ that converts the ordering $\theta_{1}$ to the ordering $\theta_{2}=(v_{2},v_{4},v_{1},v_{5},v_{3},v_{6},v_{7})$ of $V(D)$, and let \textit{operation $2$} be the permutation of the vertices $v_{1},...,v_{7}$ that converts the ordering $\theta_{1}$ to the ordering to the ordering $\theta_{3}=(v_{1},v_{2},v_{5},v_{3},v_{7},v_{4},v_{6})$ of $V(D)$.
Let $H$ be a regular flotilla-galaxy under an ordering $\theta = (v_{1},...,v_{h})$ of its vertices with $\mid$$H$$\mid$ $= h$. Let $B^{\gamma}_{1},...,B^{\gamma}_{l}$ be the $\gamma$-boats of $H$ under $\theta$.
Define $\Theta_{\theta}(H) = \lbrace \theta^{'}$ an ordering of $V(H)$; $\theta^{'}$ is obtained from $\theta$ by performing operation $1$ to the vertex set of some left $\gamma_{1}$-boats and some left $\gamma_{2}$-boats of $H$ under $\theta$, and operation $2$ to the vertex set of some right $\gamma_{1}$-boats and some right $\gamma_{2}$-boats of $H$ under $\theta$$\rbrace$. Notice that $\mid$$\Theta_{\theta}(H)$$\mid$ $= 2^{l}$. Also notice that when applying such operations to $\gamma$-boats, the $4$-vertex path and the $3$-vertex middle star will be transformed to a triangle and two $2$-vertex stars, which is very interesting (see Figure \ref{fig:gammaboats}).
\subsubsection{Corresponding digraph}
Unlike galaxies and constellations, in flotilla-galaxies, the flotilla-galaxy ordering alone failed to make the proof for flotilla-galaxies work. To this end we started thinking about another crucial orderings that give different backward arc configurations like triangles and stars. The corresponding digraph is a tool we introduce in \cite{sg} for which we construct starting from a tournament $H$ a new larger digraph following all backward arc configurations of $H$ under different crucial orderings of its vertex set. So that whatever the outcomes and all possible cases that we have to study in the proof, we will be able to extract $H$ from an $\epsilon$-critical tournament $T$ using its corresponding digraph. Note that there is not specific rule one can follow to construct the corresponding digraph for any tournament. For example the way we construct the corresponding digraph for asterisms is different than that in flotilla-galaxies, but both follow the same principle described    above.

 Let $H$ be a regular flotilla-galaxy under an ordering $\theta = (v_{1},...,v_{h})$ of its vertices with $\mid$$H$$\mid$ $= h$. Let $B^{\gamma}_{1},...,B^{\gamma}_{l}$ be the $\gamma$-boats of $H$ under $\theta$. In what follows we explain how we construct the \textit{corresponding digraph of $H$ under $\theta$}. We call this digraph the \textit{helping digraph} (or \textit{key digraph}) due to its impact on the proof of our result.\vspace{2.5mm}\\
$\ast$ Let $ 1 \leq i \leq l$ such that $B^{\gamma}_{i} = \lbrace v_{s_{i}},v_{s_{i}+1},v_{s_{i}+2},v_{s_{i}+3},v_{s_{i}+4},v_{q_{i}},v_{q_{i}+1} \rbrace$ is a left $\gamma_{1}$-boat of $H$ under $\theta$. Let $\widehat{B^{\gamma}_{i}}$ be the $13$-vertex digraph  that is obtained from $B^{\gamma}_{i}$ by deleting the arcs
$(v_{s_{i}+2},v_{q_{i}+1}),(v_{s_{i}+3},v_{q_{i}}),(v_{s_{i}},v_{s_{i}+1})$ and adding six extra vertices $x_{i}$ just after $v_{s_{i}}$, $r_{i}$ just after  $v_{s_{i}+4}$, $g_{i}$ just after $r_{i}$, $m_{i}$ just after $g_{i}$, $w_{i}$ just after $v_{s_{i}+3}$, and $y_{i}$ just after $v_{q_{i}+1}$, such that $r_{i}\leftarrow y_{i}, w_{i}\leftarrow m_{i},x_{i}\leftarrow g_{i}, V(\widehat{B^{\gamma}_{i}})\backslash\lbrace r_{i}\rbrace\rightarrow y_{i} , v_{s_{i}}\rightarrow x_{i}\rightarrow V(\widehat{B^{\gamma}_{i}})\backslash\lbrace v_{s_{i}},g_{i}\rbrace , \lbrace v_{s_{i}},v_{s_{i}+1},v_{s_{i}+2},v_{s_{i}+3}\rbrace\rightarrow w_{i}\rightarrow \lbrace v_{s_{i}+4},r_{i},g_{i},v_{q_{i}},v_{q_{i}+1}\rbrace , \lbrace v_{s_{i}},v_{s_{i}+1},v_{s_{i}+2},v_{s_{i}+3}, v_{s_{i}+4}\rbrace \rightarrow \lbrace r_{i},g_{i},m_{i}\rbrace \rightarrow \lbrace v_{q_{i}},v_{q_{i}+1}\rbrace , r_{i}\rightarrow \lbrace g_{i},m_{i}\rbrace$, and $g_{i}\rightarrow m_{i}$.
  We write $\widehat{B^{\gamma}_{i}} = \lbrace v_{s_{i}},x_{i},v_{s_{i}+1},v_{s_{i}+2},v_{s_{i}+3},w_{i},v_{s_{i}+4},$ $r_{i},g_{i},m_{i},v_{q_{i}},v_{q_{i}+1},y_{i} \rbrace$ and we call $(v_{s_{i}},x_{i},v_{s_{i}+1},v_{s_{i}+2},v_{s_{i}+3},w_{i},v_{s_{i}+4},r_{i},g_{i},m_{i},v_{q_{i}},v_{q_{i}+1},y_{i})$ the \textit{forest ordering of $\widehat{B^{\gamma}_{i}}$} and $x_{i},v_{s_{i}+1},v_{s_{i}+2},v_{s_{i}+3},w_{i},r_{i},v_{q_{i}+1}$ the \textit{leaves of} $\widehat{B^{\gamma}_{i}}$. We say that $\widehat{B^{\gamma}_{i}}$ is the \textit{mutant left $\gamma_{1}$-boat} (also we say that $\widehat{B^{\gamma}_{i}}$ is the \textit{digraph corresponding to} $B^{\gamma}_{i}$ \textit{under} $\theta$) (see Figure \ref{fig:boatdigraphs}). \vspace{1.5mm}\\
  $\ast$ Let $ 1 \leq i \leq l$ such that $B^{\gamma}_{i} = \lbrace v_{s_{i}},v_{s_{i}+1},v_{s_{i}+2},v_{s_{i}+3},v_{s_{i}+4},v_{q_{i}-1},v_{q_{i}} \rbrace$ is a left $\gamma_{2}$-boat of $H$ under $\theta$. Let $\widehat{B^{\gamma}_{i}}$ be the $13$-vertex digraph  that is obtained from $B^{\gamma}_{i}$ by deleting the arcs
$(v_{s_{i}+2},v_{q_{i}-1}),(v_{s_{i}+3},v_{q_{i}}),(v_{s_{i}},v_{s_{i}+1})$ and adding six extra vertices $x_{i}$ just after $v_{s_{i}}$, $r_{i}$ just after  $v_{s_{i}+4}$, $g_{i}$ just after $r_{i}$, $m_{i}$ just after $g_{i}$, $w_{i}$ just after $v_{s_{i}+3}$, and $y_{i}$ just before $v_{q_{i}-1}$, such that $ V(\widehat{B^{\gamma}_{i}})\backslash\lbrace r_{i},v_{q_{i}-1},v_{q_{i}}\rbrace\rightarrow y_{i}\rightarrow \lbrace  r_{i},v_{q_{i}-1},v_{q_{i}}\rbrace , \lbrace v_{s_{i}},g_{i}\rbrace\rightarrow x_{i}\rightarrow V(\widehat{B^{\gamma}_{i}})\backslash\lbrace v_{s_{i}},g_{i}\rbrace , \lbrace v_{s_{i}},...,v_{s_{i}+3},m_{i} \rbrace\rightarrow w_{i}\rightarrow \lbrace v_{s_{i}+4},r_{i},g_{i},v_{q_{i}-1},v_{q_{i}}\rbrace , \lbrace v_{s_{i}},...,v_{s_{i}+4}\rbrace \rightarrow \lbrace r_{i},g_{i},m_{i}\rbrace \rightarrow \lbrace v_{q_{i}-1},v_{q_{i}}\rbrace , r_{i}\rightarrow \lbrace g_{i},m_{i}\rbrace$, $g_{i}\rightarrow m_{i}$. 
  We write $\widehat{B^{\gamma}_{i}} = \lbrace v_{s_{i}},x_{i},v_{s_{i}+1},v_{s_{i}+2},v_{s_{i}+3},w_{i},v_{s_{i}+4},r_{i},g_{i},m_{i},y_{i},v_{q_{i}-1},$ $v_{q_{i}} \rbrace$ and we call $(v_{s_{i}},x_{i},v_{s_{i}+1},v_{s_{i}+2},v_{s_{i}+3},w_{i},v_{s_{i}+4},r_{i},g_{i},m_{i},y_{i},v_{q_{i}-1},v_{q_{i}})$ the \textit{forest ordering of $\widehat{B^{\gamma}_{i}}$} and $x_{i},v_{s_{i}+1},v_{s_{i}+2},v_{s_{i}+3},w_{i},r_{i},v_{q_{i}-1}$ the \textit{leaves of} $\widehat{B^{\gamma}_{i}}$. We say that $\widehat{B^{\gamma}_{i}}$ is the \textit{mutant left $\gamma_{2}$-boat} (also we say that $\widehat{B^{\gamma}_{i}}$ is the \textit{digraph corresponding to} $B^{\gamma}_{i}$ \textit{under} $\theta$) (see Figure \ref{fig:boatdigraphs}). \vspace{1.5mm}\\
  $\ast$ Let $ 1 \leq i \leq l$ such that $B^{\gamma}_{i} = \lbrace v_{q_{i}-1},v_{q_{i}},v_{s_{i}},v_{s_{i}+1},v_{s_{i}+2},v_{s_{i}+3},v_{s_{i}+4} \rbrace$ is a right $\gamma_{1}$-boat of $H$ under $\theta$. Let $\widehat{B^{\gamma}_{i}}$ be the $13$-vertex digraph  that is obtained from $B^{\gamma}_{i}$ by deleting the arcs
$(v_{q_{i}-1},v_{s_{i}+2}),(v_{q_{i}},v_{s_{i}+1}),(v_{s_{i}+3},$ $v_{s_{i}+4})$ and adding six extra vertices $x_{i}$ just after $v_{s_{i}+3}$, $r_{i}$ just before  $v_{s_{i}}$, $g_{i}$ just before $r_{i}$, $m_{i}$ just before $g_{i}$, $w_{i}$ just after $v_{s_{i}}$, and $y_{i}$ just before $v_{q_{i}-1}$, such that $r_{i}\rightarrow y_{i}\rightarrow V(\widehat{B^{\gamma}_{i}})\backslash\lbrace r_{i}\rbrace , V(\widehat{B^{\gamma}_{i}})\backslash\lbrace g_{i},v_{s_{i}+4}\rbrace \rightarrow x_{i}\rightarrow \lbrace v_{s_{i}+4},g_{i}\rbrace , \lbrace v_{q_{i}-1},v_{q_{i}},g_{i},r_{i},v_{s_{i}}\rbrace \rightarrow w_{i}\rightarrow  \lbrace m_{i},v_{s_{i}+1},...,v_{s_{i}+4} \rbrace , \lbrace v_{q_{i}-1},v_{q_{i}}\rbrace \rightarrow \lbrace m_{i},g_{i},r_{i}\rbrace \rightarrow \lbrace v_{s_{i}},...,v_{s_{i}+4}\rbrace ,$ $ m_{i}\rightarrow \lbrace g_{i},r_{i}\rbrace$, $g_{i}\rightarrow r_{i}$.
  We write $\widehat{B^{\gamma}_{i}} = \lbrace y_{i},v_{q_{i}-1},v_{q_{i}},m_{i},g_{i},r_{i},v_{s_{i}},w_{i},v_{s_{i}+1},v_{s_{i}+2},v_{s_{i}+3},x_{i},v_{s_{i}+4} \rbrace$ and we call $(y_{i},v_{q_{i}-1},v_{q_{i}},m_{i},g_{i},r_{i},v_{s_{i}},w_{i},v_{s_{i}+1},v_{s_{i}+2},v_{s_{i}+3},x_{i},v_{s_{i}+4})$ the \textit{forest ordering of $\widehat{B^{\gamma}_{i}}$} and $x_{i},v_{s_{i}+1},v_{s_{i}+2},$ $v_{s_{i}+3},w_{i},r_{i},v_{q_{i}-1}$ the \textit{leaves of} $\widehat{B^{\gamma}_{i}}$. We say that $\widehat{B^{\gamma}_{i}}$ is the \textit{mutant right $\gamma_{1}$-boat} (also we say that $\widehat{B^{\gamma}_{i}}$ is the \textit{digraph corresponding to} $B^{\gamma}_{i}$ \textit{under} $\theta$) (see Figure \ref{fig:boatdigraphs}). \vspace{1.5mm}\\ 
  $\ast$ Let $ 1 \leq i \leq l$ such that $B^{\gamma}_{i} = \lbrace v_{q_{i}},v_{q_{i}+1},v_{s_{i}},v_{s_{i}+1},v_{s_{i}+2},v_{s_{i}+3},v_{s_{i}+4} \rbrace$ is a right $\gamma_{2}$-boat of $H$ under $\theta$. Let $\widehat{B^{\gamma}_{i}}$ be the $13$-vertex digraph  that is obtained from $B^{\gamma}_{i}$ by deleting the arcs
$(v_{q_{i}+1},v_{s_{i}+2}),(v_{q_{i}},v_{s_{i}+1}),(v_{s_{i}+3},$ $v_{s_{i}+4})$ and adding six extra vertices $x_{i}$ just after $v_{s_{i}+3}$, $r_{i}$ just before  $v_{s_{i}}$, $g_{i}$ just before $r_{i}$, $m_{i}$ just before $g_{i}$, $w_{i}$ just after $v_{s_{i}}$, and $y_{i}$ just after $v_{q_{i}+1}$, such that $\lbrace v_{q_{i}},v_{q_{i}+1},r_{i}\rbrace\rightarrow y_{i}\rightarrow V(\widehat{B^{\gamma}_{i}})\backslash\lbrace v_{q_{i}},v_{q_{i}+1},r_{i}\rbrace , V(\widehat{B^{\gamma}_{i}})\backslash\lbrace g_{i},$ $v_{s_{i}+4}\rbrace \rightarrow x_{i}\rightarrow \lbrace v_{s_{i}+4},g_{i}\rbrace , \lbrace v_{q_{i}},v_{q_{i}+1},g_{i},r_{i},v_{s_{i}}\rbrace \rightarrow w_{i}\rightarrow  \lbrace m_{i},v_{s_{i}+1},...,v_{s_{i}+4} \rbrace , \lbrace v_{q_{i}},v_{q_{i}+1}\rbrace \rightarrow \lbrace m_{i},g_{i},r_{i}\rbrace \rightarrow \lbrace v_{s_{i}},...,v_{s_{i}+4}\rbrace , m_{i}\rightarrow \lbrace g_{i},r_{i}\rbrace$, $g_{i}\rightarrow r_{i}$. 
  We write $\widehat{B^{\gamma}_{i}} = \lbrace v_{q_{i}},v_{q_{i}+1},y_{i},m_{i},g_{i},r_{i},v_{s_{i}},w_{i},v_{s_{i}+1},v_{s_{i}+2},v_{s_{i}+3},x_{i},$ $v_{s_{i}+4} \rbrace$ and we call $(v_{q_{i}},v_{q_{i}+1},y_{i},m_{i},g_{i},r_{i},v_{s_{i}},w_{i},v_{s_{i}+1},v_{s_{i}+2},v_{s_{i}+3},x_{i},v_{s_{i}+4})$ the \textit{forest ordering of $\widehat{B^{\gamma}_{i}}$} and $x_{i},v_{s_{i}+1},v_{s_{i}+2},v_{s_{i}+3},w_{i},r_{i},v_{q_{i}+1}$ the \textit{leaves of} $\widehat{B^{\gamma}_{i}}$. We say that $\widehat{B^{\gamma}_{i}}$ is the \textit{mutant right $\gamma_{2}$-boat} (also we say that $\widehat{B^{\gamma}_{i}}$ is the \textit{digraph corresponding to} $B^{\gamma}_{i}$ \textit{under} $\theta$) (see Figure \ref{fig:boatdigraphs}). \\
  \begin{figure}[h]
	\centering
	\includegraphics[width=1.\linewidth]{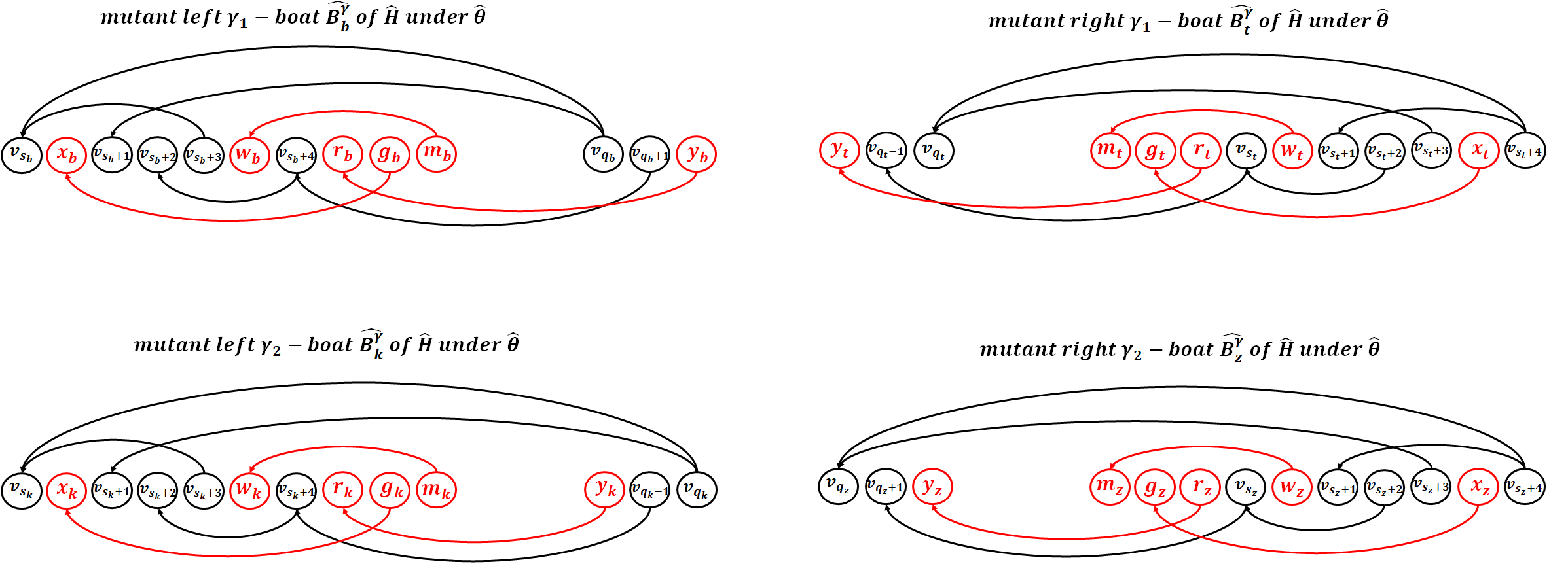}
	\caption{Mutant left $\gamma_{1}$-boat $\widehat{B^{\gamma}_{b}}$, mutant right $\gamma_{1}$-boat $\widehat{B^{\gamma}_{t}}$, mutant left $\gamma_{2}$-boat $\widehat{B^{\gamma}_{k}}$, and mutant right $\gamma_{2}$-boat $\widehat{B^{\gamma}_{z}}$. All the backward arcs are drawn. All the non-drawn arcs are forward except that the arcs $(v_{s_{b}+2},v_{q_{b}+1}),(v_{s_{b}+3},v_{q_{b}}),(v_{s_{b}},v_{s_{b}+1})\notin \widehat{B_{b}^{\gamma}}$, $(v_{s_{k}+2},v_{q_{k}-1}),(v_{s_{k}+3},v_{q_{k}}),(v_{s_{k}},v_{s_{k}+1})\notin \widehat{B_{k}^{\gamma}}$, $(v_{q_{t}-1},v_{s_{t}+2}),(v_{s_{t}+3},v_{s_{t}+4}),(v_{q_{t}},v_{s_{t}+1})\notin \widehat{B_{t}^{\gamma}}$, $(v_{q_{z}+1},v_{s_{z}+2}),(v_{s_{z}+3},v_{s_{z}+4}),(v_{q_{z}},v_{s_{z}+1})\notin \widehat{B_{z}^{\gamma}}$.}
	\label{fig:boatdigraphs}
\end{figure}
\\We are ready now to define the \textit{corresponding digraph of a flotilla-galaxy}.\vspace{2mm}\\Let $\widehat{H}$ be the digraph obtained from $H$ by replacing $B^{\gamma}_{i}$ by its corresponding digraph $\widehat{B^{\gamma}_{i}}$ for $i=1,...,l$, and let $\widehat{\theta}$ be the obtained ordering of $\widehat{H}$ (i.e $\widehat{\theta}$ is obtained from $\theta$ by replacing the vertices of $B^{\gamma}_{i}$ by the vertices of $\widehat{B^{\gamma}_{i}}$ for $i=1,...,l$ such that for all $ 1\leq i\leq l$, $\widehat{B^{\gamma}_{i}}$ has its forest ordering under $\widehat{\theta}$, and such that:\\$\bullet$  if $B^{\gamma}_{i}$ is a left $\gamma_{1}$-boat, then $v_{s_{i}},x_{i},v_{s_{i}+1},v_{s_{i}+2},v_{s_{i}+3},w_{i},v_{s_{i}+4},r_{i},g_{i},m_{i}$ are consecutive  under $\widehat{\theta}$ and $v_{q_{i}},v_{q_{i}+1},y_{i}$ are consecutive  under $\widehat{\theta}$.\\$\bullet$ if $B^{\gamma}_{i}$ a left $\gamma_{2}$-boat, then $v_{s_{i}},x_{i},v_{s_{i}+1},v_{s_{i}+2},v_{s_{i}+3},w_{i},v_{s_{i}+4},r_{i},g_{i},m_{i}$ are consecutive  under $\widehat{\theta}$ and $y_{i},v_{q_{i}-1},v_{q_{i}}$ are consecutive  under $\widehat{\theta}$.\\$\bullet$ if $B^{\gamma}_{i}$ is a right $\gamma_{1}$-boat, then $y_{i},v_{q_{i}-1},v_{q_{i}}$ are consecutive  under $\widehat{\theta}$ and $m_{i},g_{i},r_{i},v_{s_{i}},w_{i},v_{s_{i}+1},v_{s_{i}+2},v_{s_{i}+3},x_{i},$ $v_{s_{i}+4}$ are consecutive  under $\widehat{\theta}$.\\ $\bullet$ if $B^{\gamma}_{i}$ is a right $\gamma_{2}$-boat, then $v_{q_{i}},v_{q_{i}+1},y_{i}$ are consecutive  under $\widehat{\theta}$ and $m_{i},g_{i},r_{i},v_{s_{i}},w_{i},v_{s_{i}+1},v_{s_{i}+2},v_{s_{i}+3},x_{i},$ $v_{s_{i}+4}$ are consecutive  under $\widehat{\theta}$).\\ We have $V(\widehat{H}) = V(H)\cup (\bigcup_{i=1}^{l}\lbrace x_{i},w_{i},r_{i},g_{i},m_{i},y_{i}\rbrace)$ and\\ $E(\widehat{H})= (E(H)\backslash \displaystyle{\bigcup_{i=1}^{l}E(B^{\gamma}_{i}))}\cup \displaystyle{\bigcup_{i=1}^{l}}E(\widehat{B^{\gamma}_{i}}) \cup [[\displaystyle{\bigcup_{i=1}^{l}}$$(\displaystyle{\bigcup_{p\in X_{i}}}$$(\lbrace (x,p): x<_{\widehat{\theta}}p$ and $x\in V(\widehat{H}) \rbrace \cup $$\lbrace (p,x): p<_{\widehat{\theta}}x$ and $x\in V(\widehat{H}) \rbrace ))]\backslash \displaystyle{\bigcup_{i=1}^{l}} ( \lbrace(r_{i},y_{i}): r_{i}<_{\widehat{\theta}}$ $ y_{i} \rbrace \cup \lbrace(y_{i},r_{i}): y_{i}<_{\widehat{\theta}}r_{i} \rbrace \cup \lbrace(m_{i},w_{i}): m_{i}<_{\widehat{\theta}}$ $ w_{i} \rbrace \cup \lbrace(w_{i},m_{i}): w_{i}<_{\widehat{\theta}}m_{i} \rbrace \cup \lbrace(x_{i},g_{i}): x_{i}<_{\widehat{\theta}}$ $ g_{i} \rbrace \cup \lbrace(g_{i},x_{i}): g_{i}<_{\widehat{\theta}}x_{i} \rbrace )]$, where $X_{i}=\lbrace m_{i},g_{i},x_{i},w_{i},r_{i},y_{i}\rbrace$. Note that $E(B(\widehat{H},\widehat{\theta})) = E(B(H,\theta))\cup\lbrace x_{i}g_{i},w_{i}m_{i},r_{i}y_{i}: i=1,...,l\rbrace$. We say that $\widehat{H}$ is the \textit{digraph corresponding to} $H$ \textit{under} $\theta$, and $\widehat{\theta}$ is the \textit{ordering of $V(\widehat{H})$ corresponding to} $\theta$.
\subsubsection{Corresponding smooth $(c,\lambda,w)$-structure} 
Let $s$ be a $\lbrace 0,1 \rbrace$$-$vector. Denote $s_{c}$ the vector obtained from $s$ by replacing every subsequence of consecutive $1'$s by single $1$. 

Let $H$ be a regular flotilla-galaxy under an ordering $\theta$ of its vertices with $\mid$$H$$\mid$ $= h$. Let $B^{\gamma}_{1},...,B^{\gamma}_{l}$ be the $\gamma$-boats of $H$ under $\theta$, and let $Q_{1},...,Q_{l}$ be the frontier stars of $H$$\mid$$X$ under $\overline{\theta}$. Let $\widehat{\theta} = (u_{1},...,u_{h+6l})$ be the ordering of $V(\widehat{H})$ corresponding to $\theta$. For $i \in \lbrace 0,...,l \rbrace$ define $\widehat{H^{i}} = \widehat{H}$$\mid$$\bigcup_{j=1}^{i}(V(\widehat{B^{\gamma}_{j}})\cup V(Q_i))$ where $\widehat{H^{l}} = \widehat{H}$, and $\widehat{H^{0}}$ is the empty digraph. 
Let $s^{\widehat{H},\widehat{\theta}}$ be a $\lbrace 0,1 \rbrace$$-$vector such that $s^{\widehat{H},\widehat{\theta}}(i) = 1$ if and only if $u_{i}$ is a leaf of one of the stars  of $\widehat{H}$ under $\widehat{\theta}$ or a leaf of one of $\widehat{B^{\gamma}_{j}}$ of $\widehat{H}$ under $\widehat{\theta}$ for $j=1,...,l$. Let $w = s_{c}^{\widehat{H},\widehat{\theta}}$ and let $i_{r}$ be such that $w_{i_{r}}=1$. Let $j$ be such that $s^{\widehat{H},\widehat{\theta}}_{j}=1$. We say that $s^{\widehat{H},\widehat{\theta}}_{j}$ \textit{corresponds to} $w_{i_{r}}$ if $s^{\widehat{H},\widehat{\theta}}_{j}$ belongs to the subsequence of consecutive $1'$s that is replaced by the entry $w_{i_{r}}$. For $k \in \lbrace 1,...,l \rbrace$ let $\widehat{\theta}_{k} = (v_{k_{1}},...,v_{k_{t_{k}}})$ with $t_{k}=$ $\mid$$ \widehat{H^{k}}$$ \mid$, be the restriction of $\widehat{\theta}$ to $V(\widehat{H}^{k})$. Let $s^{\widehat{H},\widehat{\theta}}_{\widehat{H}^{k}}$ be the restriction of $s^{\widehat{H},\widehat{\theta}}$ to the $0's$ and $1's$ corresponding to $V(\widehat{H}^{k})$ (notice that $s^{\widehat{H},\widehat{\theta}}_{\widehat{H}^{k}}= s^{\widehat{H}^{k},\widehat{\theta}_{k}}$) and let $^{c}s^{\widehat{H},\widehat{\theta}}_{\widehat{H}^{k}}$ be the vector obtained from $s^{\widehat{H},\widehat{\theta}}_{\widehat{H}^{k}}$ by replacing every subsequence of consecutive $1's$ corresponding to the same entry of $s^{\widehat{H},\widehat{\theta}}_{c}$ by single $1$. We say that a smooth $(c,\lambda ,w)$-structure of a tournament $T$ \textit{corresponds}  \textit{to $\widehat{H}^{k}$ under $(\widehat{H},\widehat{\theta})$} if $w =$ $ ^{c}s^{\widehat{H},\widehat{\theta}}_{\widehat{H}^{k}}$. Notice that $s^{\widehat{H},\widehat{\theta}}_{\widehat{H}^{l}}=s^{\widehat{H},\widehat{\theta}}$ and $^{c}s^{\widehat{H},\widehat{\theta}}_{\widehat{H}^{l}}=s^{\widehat{H},\widehat{\theta}}_{c}$.\\
Let $\nu =$ $^{c}s^{\widehat{H},\widehat{\theta}}_{\widehat{H}^{k}}$. Let $\delta^{\nu}:$ $\lbrace j: \nu_{j} = 1 \rbrace \rightarrow \mathbb{N}$ be a function that assigns to every nonzero entry of $\nu$ the number of consecutive $1'$s of $s^{\widehat{H},\widehat{\theta}}_{\widehat{H}^{k}}$ replaced by that entry of $\nu$.
Fix $k \in \lbrace 0,...,l \rbrace$. Let $(S_{1},...,S_{\mid w \mid})$ be a smooth $(c,\lambda ,w)$-structure corresponding to $\widehat{H}^{k}$ under $(\widehat{H},\widehat{\theta})$.
 Let $i_{r}$ be such that $w(i_{r}) = 1$. Assume that $S_{i_{r}} = \lbrace s^{1}_{i_{r}},...,s_{i_{r}}^{\mid S_{i_{r}} \mid} \rbrace$ and $(s^{1}_{i_{r}},...,s_{i_{r}}^{\mid S_{i_{r}} \mid})$ is a transitive ordering. Write $m(i_{r}) = \lfloor\frac{\mid S_{i_{r}} \mid}{\delta^{w}(i_{r})}\rfloor$. Denote $S^{j}_{i_{r}} = \lbrace s^{(j-1)m(i_{r})+1}_{i_{r}},...,s_{i_{r}}^{jm(i_{r})} \rbrace$ for $j \in \lbrace 1,...,\delta^{w}(i_{r}) \rbrace$. For every $v \in S^{j}_{i_{r}}$ denote $\xi(v) = (\mid$$\lbrace k < i_{r}: w(k) = 0 \rbrace$$\mid$ $+$ $\displaystyle{\sum_{k < i_{r}: w(k) = 1}\delta^{w}(k) })$ $+$ $j$. For every $v \in S_{i_{r}}$ such that $w(i_{r}) = 0$ denote $\xi(v) = (\mid$$\lbrace k < i_{r}: w(k) = 0 \rbrace$$\mid$ $+$ $\displaystyle{\sum_{k < i_{r}: w(k) = 1}\delta^{w}(k) })$ $+$ $1$. We say that $\widehat{H}^{k}$ is \textit{well-contained in} $(S_{1},...,S_{\mid w \mid})$ that corresponds to $\widehat{H}^{k}$ under $(\widehat{H},\widehat{\theta})$ if there is an injective homomorphism $f$ of $\widehat{H}^{k}$ into $T$$\mid$$\bigcup_{i = 1}^{\mid w \mid}S_{i}$ such that $\xi(f(u_{k_j})) = j$ for every $j \in \lbrace 1,...,t_{k} \rbrace$, where $\widehat{\theta}_{k} = (u_{k_{1}},...,u_{k_{t_{k}}})$. 
 \subsection{Proof}
 \begin{theorem}\label{gftheorem}
Let $H$ be a regular generalized flotilla-galaxy under an ordering $\theta$ of its vertices with $\mid$$H$$\mid$ $= h$. Let $B^{\gamma}_{1},...,B^{\gamma}_{l}$ be the $\gamma$-boats of $H$ under $\theta$ and let $Q_{1},...,Q_{l}$ be the frontier stars of $H$$\mid$$X$ under $\overline{\theta}$. Let $0 < \lambda < \frac{1}{(4h)^{h+4}}$, $c > 0$ be constants, and $w$ be a $\lbrace 0,1 \rbrace$-vector. Fix $k \in \lbrace 0,...,l \rbrace$ and let $\widehat{\lambda} = (2h)^{l-k}\lambda$ and $\widehat{c} = \frac{c}{(2h)^{l-k}}$. There exist $ \epsilon_{k} > 0$ such that $\forall 0 < \epsilon < \epsilon_{k}$, for every $\epsilon$-critical tournament $T$ with $\mid$$T$$\mid$ $= n$ containing $\chi = (S_{1},...,S_{\mid w \mid})$ as a smooth $(\widehat{c},\widehat{\lambda},w)$-structure corresponding to $\widehat{H^{k}}$ under  $(\widehat{H},\widehat{\theta})$, we have $\widehat{H^{k}}$ is well-contained in $\chi$.  
\end{theorem}
\begin{proof}
The proof is by induction on $k$. For $k=0$ the statement is obvious since $\widehat{H^{0}}$ is the empty digraph. Suppose that $\chi = (S_{1},...,S_{\mid w \mid})$ is a smooth $(\widehat{c},\widehat{\lambda},w)$-structure in $T$ corresponding to $\widehat{H^{k}}$ under $(\hat{H},\hat{\theta})$, with $\hat{\theta}=(h_1,...,h_{h+6l})$. Let $\widehat{\theta}_{k} = (h_{k_{1}},...,h_{k_{s}})$ be the restriction of $\widehat{\theta}$ to $V(\widehat{H^{k}})$, where $s=$ $\mid$$ \widehat{H^{k}}$$ \mid$. Let $\widehat{B^{\gamma}_{k}} = \lbrace h_{k_{q_{1}}},...,h_{k_{q_{13}}} \rbrace$. Let $h_{k_{p_{0}}}$ be the center of $Q_{k}$ and $h_{k_{p_{1}}},...,h_{k_{p_{q}}}$ be its leafs for some integer $q>0$. Let $D_{i} = \lbrace v \in \bigcup_{j=1}^{\mid w \mid}S_{j};$ $ \xi(v) = q_{i} \rbrace$ for $i=1,...,13$. Assume that $B^{\gamma}_{k}$ is a left $\gamma_{1}$$-$boat (else, the argument is similar, and we omit it). Then there exist $ 1 \leq y \leq \mid$$w$$\mid$ and $ 1 \leq b \leq \mid$$w$$\mid$, such that $y+5< b$ and $D_{1} = S_{y}$, $D_{2} = S^{1}_{y+1}$, $D_{3} = S^{2}_{y+1}$, $D_{4} = S^{3}_{y+1}$, $D_{5} = S^{4}_{y+1}$, $D_{6} = S^{5}_{y+1}$, $D_{7} = S_{y+2}$, $D_{8} = S_{y+3}$, $D_{9} = S_{y+4}$, $D_{10} = S_{y+5}$, $D_{11} = S_{b}$, $D_{12} = S_{b+1}$, $D_{13} = S_{b+2}$ with $w(y)=w(y+2)=w(y+4)=w(y+5)=w(b) =w(b+2)= 0$ and $w(y+1) =w(y+3)=w(b+1)= 1$. By Lemma \ref{s}, since $T$ is $\epsilon$-critical and $\epsilon < min\lbrace log_{\frac{\widehat{c}}{4}}(\frac{1}{2}), log_{\frac{\widehat{c}}{2}}(1-\frac{\widehat{c}}{6})\rbrace$, there exist vertices $x_{1}\in D_{1}$, $x_{3}\in D_{3}$, $x_{5}\in D_{5}$, $x_{11}\in D_{11}$, such that $\lbrace x_{1},x_{3}\rbrace \leftarrow x_{11}$ and $x_{1}\leftarrow x_{5}$. Let $D_{4}^{*} = \lbrace x_{4}\in D_{4}; x_{1}\rightarrow x_{4} \rightarrow x_{11} \rbrace$, $D_{7}^{*} = \lbrace x_{7}\in D_{7}; \lbrace x_{1},x_{3},x_{5}\rbrace \rightarrow x_{7} \rightarrow x_{11} \rbrace$, and $D_{12}^{*} = \lbrace x_{12}\in D_{12}; \lbrace x_{1},x_{3},x_{5},x_{11}\rbrace \rightarrow x_{12} \rbrace$. Then by Lemma \ref{g}, $\mid$$D_{7}^{*}$$\mid$ $\geq (1-4\widehat{\lambda})\widehat{c}n\geq \frac{\widehat{c}}{2}n$ since $\widehat{\lambda} \leq \frac{1}{8}$, $\mid$$D_{12}^{*}$$\mid$ $\geq (1-4\widehat{\lambda})\widehat{c}tr(T)\geq \frac{\widehat{c}}{2}tr(T)$ since $\widehat{\lambda} \leq \frac{1}{8}$, and $\mid$$D_{4}^{*}$$\mid$ $\geq \frac{1-12\widehat{\lambda}}{6}\widehat{c}tr(T)\geq \frac{\widehat{c}}{12}tr(T)$ since $\widehat{\lambda} \leq \frac{1}{24}$. 
 Since we can assume that $\epsilon < log_{\frac{\widehat{c}}{8}}(1-\frac{\widehat{c}}{12})$, then Lemma \ref{middle} implies that there exist vertices $x_{4}\in D_{4}^{*}$, $x_{7}\in D_{7}^{*}$, and $x_{12}\in D_{12}^{*}$ such that $x_{4}\leftarrow x_{7} \leftarrow x_{12}$. Let $D_{2}^{*} = \lbrace x_{2}\in D_{2}; x_{1}\rightarrow x_{2} \rightarrow \lbrace x_{7},x_{11},x_{12}\rbrace \rbrace$ and $D_{9}^{*} = \lbrace x_{9}\in D_{9}; \lbrace x_{1},x_{3},x_{4},x_{5},x_{7}\rbrace\rightarrow x_{9} \rightarrow \lbrace x_{11},x_{12}\rbrace \rbrace$. Then by Lemma \ref{g}, $\mid$$D_{2}^{*}$$\mid$ $\geq \frac{1-24\widehat{\lambda}}{6}\widehat{c}tr(T)\geq \frac{\widehat{c}}{12}tr(T)$ since $\widehat{\lambda} \leq \frac{1}{48}$ and $\mid$$D_{9}^{*}$$\mid$ $\geq (1-7\widehat{\lambda})\widehat{c}n\geq \frac{\widehat{c}}{2}n$ since $\widehat{\lambda} \leq \frac{1}{14}$. Since $\epsilon < log_{\frac{\widehat{c}}{2}}(1-\frac{\widehat{c}}{12})$, then Lemma \ref{f} implies that there exist vertices $x_{2}\in D_{2}^{*}$ and $x_{9}\in D_{9}^{*}$ such that $x_{2} \leftarrow x_{9}$. Let $D_{6}^{*} = \lbrace x_{6}\in D_{6}; x_{1}\rightarrow x_{6} \rightarrow \lbrace x_{7},x_{9},x_{11},x_{12}\rbrace \rbrace$ and $D_{10}^{*} = \lbrace x_{10}\in D_{10}; \lbrace x_{1},...,x_{5},x_{7},x_{9}\rbrace\rightarrow x_{10} \rightarrow \lbrace x_{11},x_{12}\rbrace \rbrace$. Then by Lemma \ref{g}, $\mid$$D_{6}^{*}$$\mid$ $\geq \frac{1-30\widehat{\lambda}}{6}\widehat{c}tr(T)\geq \frac{\widehat{c}}{12}tr(T)$ since $\widehat{\lambda} \leq \frac{1}{60}$ and $\mid$$D_{10}^{*}$$\mid$ $\geq (1-9\widehat{\lambda})\widehat{c}n\geq \frac{\widehat{c}}{2}n$ since $\widehat{\lambda} \leq \frac{1}{18}$. Since $\epsilon < log_{\frac{\widehat{c}}{2}}(1-\frac{\widehat{c}}{12})$, then Lemma \ref{f} implies that there exist vertices $x_{6}\in D_{6}^{*}$ and $x_{10}\in D_{10}^{*}$ such that $x_{6} \leftarrow x_{10}$. Let $D_{8}^{*} = \lbrace x_{8}\in D_{8}; \lbrace x_{1},x_{2},x_{3},x_{4},x_{5},x_{6},x_{7}\rbrace\rightarrow x_{8} \rightarrow \lbrace x_{9},x_{10},x_{11},x_{12}\rbrace \rbrace$ and $D_{13}^{*} = \lbrace x_{13}\in D_{13}; \lbrace x_{1},...,x_{7},x_{9},...,x_{12}\rbrace\rightarrow x_{13} \rbrace$. Then by Lemma \ref{g}, $\mid$$D_{8}^{*}$$\mid$ $\geq (1-11\widehat{\lambda})\widehat{c}tr(T)\geq \frac{\widehat{c}}{2}tr(T)$ since $\widehat{\lambda} \leq \frac{1}{22}$ and $\mid$$D_{13}^{*}$$\mid$ $\geq (1-11\widehat{\lambda})\widehat{c}n\geq \frac{\widehat{c}}{2}n$ since $\widehat{\lambda} \leq \frac{1}{22}$. Since $\epsilon < log_{\frac{\widehat{c}}{2}}(1-\frac{\widehat{c}}{2})$, then Lemma \ref{f} implies that there exist vertices $x_{8}\in D_{8}^{*}$ and $x_{13}\in D_{13}^{*}$ such that $x_{8} \leftarrow x_{13}$.  Then $T$$\mid$$\lbrace x_{1},...x_{13}\rbrace$ contains a copy of $\widehat{H^{k}}$$\mid$$V(\widehat{B^{\gamma}_{k}})$ where $(x_{1},...,x_{13})$ is its forest ordering. Denote this copy by $W$. For all $ 0 \leq i \leq q$, let $R_{i} = \lbrace v \in \bigcup_{j=1}^{\mid w \mid}S_{j};$ $ \xi(v) = p_{i} \rbrace$ and let $R_{i}^{*} = \bigcap_{x\in V(W)}R_{i,x}$.
Then there exist $ m,f \in \lbrace 1,...,\mid$$w$$\mid \rbrace \backslash \lbrace y,...,y+5,b,b+1,b+2 \rbrace$ with $w(m)=0$ and $w(f)=1$, such that $R_{0} = S_{m}$ and for all $ 1 \leq i \leq q$, $R_{i} \subseteq S_{f}$. Then by Lemma \ref{g}, $\mid$$R_{0}^{*}$$\mid$ $\geq (1-13\widehat{\lambda})\mid$$R_{0}$$\mid$ $\geq \frac{\mid R_{0}\mid}{2}$ $\geq\frac{\widehat{c}}{2}n$ since $\widehat{\lambda} \leq \frac{1}{26}$, and $\mid$$ R_{i}^{*}$$\mid$ $ \geq \frac{1-13h\widehat{\lambda}}{h}\mid$$ S_{f}$$ \mid$ $\geq \frac{\widehat{c}}{2h}tr(T)$ since $\widehat{\lambda} \leq \frac{1}{26h}$. Since we can assume that $\epsilon < log_{\frac{\widehat{c}}{4h}}(1-\frac{\widehat{c}}{2h})$, then by Lemma \ref{r} there exists vertices $r_{0},r_{1},...,r_{q}$ such that $r_{i} \in R_{i}^{*}$ for $i=0,1,...,q$ and \\
$\ast$ $r_{1},...,r_{q}$ are all adjacent from $r_{0}$ if $m>f$.\\
$\ast$ $r_{1},...,r_{q}$ are all adjacent to $r_{0}$ if $m<f$.\\
So $T$$\mid$$\lbrace x_{1},...,x_{13},r_{0},r_{1},...,r_{q} \rbrace$ contains a copy of $\widehat{H^{k}}$$\mid$$(V(\widehat{B^{\gamma}_{k}})\cup V(Q_{k}))$. Denote this copy by $Y$.    
 For all $ i \in \lbrace 1,...,\mid$$w$$\mid \rbrace \backslash \lbrace y,...,y+5,b,b+1,b+2,m,f \rbrace$, let $S_{i}^{*} = \bigcap_{x\in V(Y)}S_{i,x}$.  Then by Lemma \ref{g}, $\mid$$S_{i}^{*}$$\mid$ $\geq (1-\mid$$Y$$\mid\widehat{\lambda})\mid$$S_{i}$$\mid$ $\geq (1-(h+6)\widehat{\lambda})\mid$$S_{i}$$\mid$ $\geq \frac{\mid S_{i}\mid}{2h}$ since $\widehat{\lambda} \leq \frac{2h-1}{2h(h+6)}$.
Write $\mathcal{H} = \lbrace 1,...,s \rbrace \backslash \lbrace q_{1},...,q_{13},p_{0},...,p_{q} \rbrace$. If $\lbrace v\in S_{f}: \xi(v) \in \mathcal{H} \rbrace \neq \phi$, then define $J_{f} = \lbrace \eta \in \mathcal{H}: \exists v \in S_{f}$ and $\xi(v)= \eta \rbrace$. Now for all $ \eta \in J_{f}$, let $S_{f}^{*\eta}= \lbrace v \in S_{f}: \xi(v)=\eta$ and $v \in \displaystyle{\bigcap_{x\in V(Y)\backslash\lbrace r_{1},...,r_{q} \rbrace}S_{f,x}} \rbrace$. Then by Lemma \ref{g}, for all $ \eta \in J_{f}$, we have $\mid$$S_{f}^{*\eta}$$\mid$ $ \geq \frac{1-14h\widehat{\lambda}}{h}\mid $$S_{f}$$\mid $ $\geq \frac{\mid S_{f}\mid}{2h}$ since $\widehat{\lambda} \leq \frac{1}{28h}$. Now for all $ \eta \in J_{f}$, select arbitrary $\lceil \frac{\mid S_{f}\mid}{2h}\rceil$ vertices of $S_{f}^{*\eta}$ and denote the union of these $\mid$$J_{f}$$\mid$ sets by $S_{f}^{*}$.   
So we have defined $t$ sets $S_{1}^{*},...,S^{*}_{t}$, where $t = \mid$$w$$\mid$ $-10$ if $S_{f}^{*}$ is defined and $t = \mid$$w$$\mid$ $-11$ if $S_{f}^{*}$ is not defined. We have $\mid$$S_{i}^{*}$$\mid$ $\geq \frac{\widehat{c}}{2h}tr(T)$ for every defined $S_{i}^{*}$ with $w(i) = 1$, and $\mid$$S_{i}^{*}$$\mid$ $\geq \frac{\widehat{c}}{2h}n$ for every defined $S_{i}^{*}$ with $w(i) = 0$. Now Lemma \ref{b} implies that $\chi^{*}=(S_{1}^{*},...,S^{*}_{t})$ form a smooth $(\frac{\widehat{c}}{2h},2h\widehat{\lambda},w^{*})$$-$structure of $T$ corresponding to $\widehat{H^{k-1}}$ under  $(\hat{H},\hat{\theta})$, where $\frac{\widehat{c}}{2h}= \frac{c}{(2h)^{l-(k-1)}}, 2h\widehat{\lambda}=(2h)^{l-(k-1)}\lambda$, and $w^{*}$ is an appropriate $\lbrace 0,1 \rbrace$$-$vector.
Now take $\epsilon_{k} < min \lbrace \epsilon_{k-1}, log_{\frac{\widehat{c}}{4}}(\frac{1}{2}), log_{\frac{\widehat{c}}{8}}(1-\frac{\widehat{c}}{16}), log_{\frac{\widehat{c}}{4h}}(1-\frac{\widehat{c}}{2h}) \rbrace$. So by induction hypothesis $\widehat{H^{k-1}}$ is well-contained in $\chi^{*}$. Now by merging the well-contained copy of $\widehat{H^{k-1}}$ and $Y$ we get a copy of $\widehat{H^{k}}$. $\blacksquare$
\end{proof}
\begin{corollary}\label{gfcoroll}  
Let $H$ be a regular flotilla-galaxy under an ordering $\theta$ of its vertices. Let $B^{\gamma}_{1},...,B^{\gamma}_{l}$ be the $\gamma$-boats of $H$ under $\theta$, and let $Q_{1},...,Q_{l}$ be the frontier stars of $H$$\mid$$(V(H)\backslash \bigcup_{i=1}^{l}V(B^{\gamma}_{i}))$ under $\overline{\theta}$. Let $ \lambda > 0$ ($\lambda$ is small enough), $c > 0$ be constants, and let $w$ be a $\lbrace 0,1 \rbrace$-vector. Suppose that  $\chi = (S_{1},...,S_{\mid w \mid})$ is a smooth $(c,\lambda ,w)$-structure of an $\epsilon$-critical tournament $T$ ($\epsilon$ is small enough) corresponding to $\widehat{H}$ under $(\widehat{H},\widehat{\theta})$. Then $T$ contains $H$.
\end{corollary}  
\begin{proof}
$\widehat{H}=\widehat{H^{l}}$ is well-contained in $\chi$ by the previous theorem when taking $k=l$. For all $ 1 \leq i \leq l$, let $\widetilde{B_{i}^{\gamma}}= \lbrace x_{i},d_{i},b_{i},u_{i},n_{i},r_{i},p_{i},q_{i},z_{i},s_{i},f_{i},a_{i},t_{i} \rbrace$ be the copy of $\widehat{B_{i}^{\gamma}}$ in $T$, and let $\tilde{Q}_i$ be the copy of $Q_i$ in $T$. Let $\theta^{'}$ be the ordering of $A = \bigcup_{i=1}^{l}(V(\widetilde{B_{i}^{\gamma}})\cup V(\tilde{Q}_i))$ according to their appearance in $(S_{1},...,S_{\mid w \mid})$ (that is if $a,b \in A $ and $a \in S_{i}$, $b\in S_{j}$ with $i<j$ then $a$ precedes $b$ in $\theta^{'}$, and if $a\in S_{j}^{m},b\in S_{j}^{r}$ with $m<r$ then $a$ precedes $b$ in $\theta^{'}$). Let $ 1 \leq i \leq l$ such that $B_{i}^{\gamma}$ is a left $\gamma_{1}$$-$boat. If $u_{i} \leftarrow a_{i}$, then we remove $x_{i},d_{i},b_{i},n_{i},z_{i},f_{i}$ from $\theta^{'}$. Otherwise, if $x_{i} \leftarrow b_{i}$, then we remove $u_{i},n_{i},r_{i},p_{i},s_{i},a_{i}$ from $\theta^{'}$. Otherwise, if $n_{i} \leftarrow f_{i}$, then we remove $b_{i},u_{i},r_{i},p_{i},s_{i},a_{i}$ from $\theta^{'}$. Otherwise, $u_{i} \rightarrow a_{i}$, $x_{i} \rightarrow b_{i}$ and $n_{i} \rightarrow f_{i}$, in this case we remove $d_{i},r_{i},q_{i},z_{i},s_{i},t_{i}$ from $\theta^{'}$. Note that in the first three cases we obtain the cyclic ordering of the left $\gamma_{1}$-boat and in the last case we obtain the forest ordering of the left $\gamma_{1}$-boat. Let $ 1 \leq i \leq l$ such that $B_{i}^{\gamma}$ is a left $\gamma_{2}$-boat. If $u_{i} \leftarrow a_{i}$, then we remove $x_{i},d_{i},b_{i},n_{i},z_{i},t_{i}$ from $\theta^{'}$. Otherwise, if $x_{i} \leftarrow b_{i}$, then we remove $u_{i},n_{i},r_{i},p_{i},s_{i},a_{i}$ from $\theta^{'}$. Otherwise, if $n_{i} \leftarrow t_{i}$, then we remove $b_{i},u_{i},r_{i},p_{i},s_{i},a_{i}$ from $\theta^{'}$. Otherwise, $u_{i} \rightarrow a_{i}$, $x_{i} \rightarrow b_{i}$ and $n_{i} \rightarrow t_{i}$, in this case we remove $d_{i},r_{i},q_{i},z_{i},s_{i},f_{i}$ from $\theta^{'}$. Note that in the first three cases we obtain the cyclic ordering of the left $\gamma_{2}$-boat and in the last case we obtain the forest ordering of the left $\gamma_{2}$-boat. Let $ 1 \leq i \leq l$ such that $B_{i}^{\gamma}$ is a right $\gamma_{1}$-boat. If $d_{i} \leftarrow s_{i}$, then we remove $b_{i},n_{i},z_{i},f_{i},a_{i},t_{i}$ from $\theta^{'}$. Otherwise, if $b_{i} \leftarrow z_{i}$, then we remove $d_{i},u_{i},p_{i},q_{i},s_{i},f_{i}$ from $\theta^{'}$. Otherwise, if $f_{i} \leftarrow t_{i}$, then we remove $d_{i},u_{i},p_{i},q_{i},z_{i},s_{i}$ from $\theta^{'}$. Otherwise, $d_{i} \rightarrow s_{i}$, $b_{i} \rightarrow z_{i}$ and $f_{i} \rightarrow t_{i}$, in this case we remove $x_{i},u_{i},n_{i},r_{i},q_{i},a_{i}$ from $\theta^{'}$. Note that in the first three cases we obtain the cyclic ordering of the right $\gamma_{1}$$-$boat and in the last case we obtain the forest ordering of the right $\gamma_{1}$-boat. Let $ 1 \leq i \leq l$ such that $B_{i}^{\gamma}$ is a right $\gamma_{2}$-boat. If $d_{i} \leftarrow s_{i}$, then we remove $x_{i},n_{i},z_{i},f_{i},a_{i},t_{i}$ from $\theta^{'}$. Otherwise, if $x_{i} \leftarrow z_{i}$, then we remove $d_{i},u_{i},p_{i},q_{i},s_{i},f_{i}$ from $\theta^{'}$. Otherwise, if $f_{i} \leftarrow t_{i}$, then we remove $d_{i},u_{i},p_{i},q_{i},z_{i},s_{i}$ from $\theta^{'}$. Otherwise, $d_{i} \rightarrow s_{i}$, $x_{i} \rightarrow z_{i}$ and $f_{i} \rightarrow t_{i}$, in this case we remove $b_{i},u_{i},n_{i},r_{i},q_{i},a_{i}$ from $\theta^{'}$. Note that in the first three cases we obtain the cyclic ordering of the right $\gamma_{2}$-boat and in the last case we obtain the forest ordering of the right $\gamma_{2}$$-$boat. We apply this rule for all $ 1 \leq i \leq l$. We obtain one of the orderings in $\Theta_{\theta}(H)$. So $T$ contains $H$. $\blacksquare$ \vspace{2mm}\\   
\end{proof}
We are ready to prove Theorem \ref{generalflotilla-galaxy}:\vspace{2mm}\\
\begin{proof}
Let $H$ be a flotilla-galaxy under $\theta$. We may assume that $H$ is a regular flotilla-galaxy since every flotilla-galaxy is a subtournament of a regular  flotilla-galaxy. Let $B^{\gamma}_{1},...,B^{\gamma}_{l}$ be the $\gamma$-boats of $H$ under $\theta$, and let $Q_{1},...,Q_{l}$ be the frontier stars of $H$$\mid$$(V(H)\backslash \bigcup_{i=1}^{l}V(B^{\gamma}_{i}))$ under $\overline{\theta}$. Let $\epsilon > 0$ be small enough and let $\lambda > 0$ be small enough. Assume that $H$ does not satisfy $EHC$, then there exists an $H$-free $\epsilon$-critical tournament $T$. By Theorem \ref{i}, $T$ contains a smooth $(c,\lambda ,w)$-structure $(S_{1},...,S_{\mid w \mid})$ corresponding to $\widehat{H}$ under $(\widehat{H},\widehat{\theta})$ for some $c >0$ and appropriate $\lbrace 0,1 \rbrace$-vector $w$. Then by the previous corollary, $T$ contains $H$, a contradiction. $\blacksquare$
\end{proof}

\begin{theorem}
If $H$ is a subtournament of a flotilla-galaxy, then $H$ has the Erd\"{o}s-Hajnal property.
\end{theorem}
\begin{proof}
The result follows from Theorem \ref{generalflotilla-galaxy} and the fact that the Erd\"{o}s-Hajnal property is a hereditary property. $\blacksquare$ 
\end{proof}

\end{document}